\documentclass[reqno]{amsart}

\usepackage[all]{xy}
\usepackage{amssymb}

\newcommand{\df}{:=} 
\newcommand{\gr}{\bullet} 
\newcommand{\incl}{\hookrightarrow} 
\newcommand{\iso}{\simeq} 
\newcommand{\To}{\longrightarrow} 
\newcommand{\ieme}{\mathrm{th}} 
\newcommand{\triv}{\mathbf{1}} 
\newcommand{\sg}{\mathrm{sgn}} 

\newcommand{\el}{\mathbf{a}} 
\newcommand{\Aff}{\mathbf{A}} 
\newcommand{\dd}{\mathbf{d}} 
\newcommand{\CC}{\mathbf{C}} 
\newcommand{\HH}{\mathcal{H}} 
\newcommand{\NN}{\mathcal{N}} 
\newcommand{\ZZ}{\mathbf{Z}} 
\newcommand{\PP}{\mathbf{P}} 
\newcommand{\VV}{\mathcal{V}} 
\newcommand{\XX}{\mathcal{X}} 

\newcommand{\A}{\operatorname{A}} 
\newcommand{\der}{\mathrm{d}} 
\newcommand{\GL}{\operatorname{GL}} 
\newcommand{\h}{\operatorname{H}} 
\newcommand{\td}{\operatorname{td}} 
\newcommand{\ch}{\operatorname{ch}} 
\newcommand{\ct}{\operatorname{ct}} 
\newcommand{\Ker}{\operatorname{Ker}} 
\newcommand{\M}{\operatorname{M}} 
\newcommand{\PGL}{\operatorname{PGL}} 
\newcommand{\tr}{\operatorname{tr}} 
\newcommand{\Tt}{\mathcal{T}} 
\newcommand{\prim}{\mathrm{pr}} 

\newtheorem{theorem}{Theorem}[section]
\newtheorem{proposition}[theorem]{Proposition}
\newtheorem{lemma}[theorem]{Lemma}
\newtheorem{corollary}[theorem]{Corollary}

\begin{document}

\title[Cohomology of symmetric hypersurfaces]{Representations on the cohomology of smooth projective hypersurfaces with symmetries}

\author{Gabriel Ch\^enevert}

\address{Math. Instituut \\ Universiteit Leiden \\ Postbus 9512 \\ 2300 RA Leiden \\ Nederland}
\curraddr{IS\'EN (Universit\'e Catholique de Lille) \\ 41 Vauban \\ 59046 Lille Cedex \\ France}
\email{gabriel.chenevert@isen.fr}

\subjclass[2000]{Primary 14Q10, 19L10, 20C30}

\begin{abstract} This paper is concerned with the primitive cohomology of a smooth projective hypersurface considered as a linear representation for its automorphism group. Using the Lefschetz-Riemann-Roch formula, the character of this representation is described on each piece of the Hodge decomposition. A consequence about the existence of smooth symmetric hypersurfaces that are stable under the standard irreducible permutation representation of the symmetric group on homogeneous coordinates is drawn. \end{abstract}

\maketitle

If $X$ is a topological space admitting an action (on the right) by a finite group $G$ of automorphisms, then its cohomology $\h^\gr(X,\CC)$ naturally inherits the structure of a graded linear representation for $G$, whose (graded) character can in principle be described using the topological Lefschetz fixed point formula in terms of intersection multiplicities \cite{Nakaoka}. When $X$ is a smooth projective complex manifold and $G$ acts holomorphically, the induced action on the cohomology of $X$ preserves the Hodge decomposition
$$ \h^k(X,\CC) = \bigoplus_{p + q = k} \h^{p,q}(X, \CC). $$
The Dolbeault isomorphism interprets $ \h^{p,q}(X,\CC)$ as the coherent sheaf cohomology group $\h^q(X,\Omega_X^p)$, hence the Lefschetz-Riemann-Roch formula \cite{Donovan} may be used to describe the $p$-part of the representation. We follow this approach in the first section to derive a generating series expression for the character of the $p$-part provided the normal bundles of the loci of fixed points may be expressed as ``polynomials'' in hyperplane line bundles (Theorem \ref{main_formula}).

In the case of a smooth hypersurface, the Lefschetz hyperplane theorem guarantees that all the cohomology of $X$ is inherited from that of the ambient projective space except for the \emph{primitive} part that lives in middle degree. Accordingly, any finite group $G$ of projective transformations preserving $X$ acts trivially on non-primitive cohomology, and the machinery of the first section may be used to describe rather explicitly its linear representation on the Hodge pieces of the primitive part (Theorem \ref{th_main}). In particular, the character on the whole primitive cohomology may be described as follows: if $X$ has degree $d$ and dimension $n$, then the trace of the linear map induced by $\sigma \in G$ on the primitive cohomology of $X$ is given by
$$ \tr(\sigma^*, \h^n_\prim(X)) = \frac{(-1)^n}{d} \sum_{\alpha^d = 1} (1 - d)^{m_\alpha}, $$
where $m_\alpha$ denotes the multiplicity of $\alpha$ as an eigenvalue of a suitable linear representative $\widetilde{\sigma}$ of $\sigma$ (Corollary \ref{cor_nice}).

Consider the standard \emph{irreducible} permutation representation of the symmetric group $S_{n+3}$ on the homogeneous coordinates of $\PP^{n+1}$. We derive from the previous formula an elegant proof of the following fact (Theorem \ref{th_existence}):
\begin{quote} For $n \geq 1$ and $d \geq 2$, there exists a \emph{smooth} projective hypersurface of dimension $n$ and degree $d$ stable under the action of $S_{n+3}$ if and only if $n+3$ is coprime with $d-1$. \end{quote}

Even though such hypersurfaces have already received some attention (e.g. \cite{Shreve}), this simple condition for their existence does not seem to have been noticed before.

The third section is devoted to a proof of the underlying representation-theoretic statement (Theorem \ref{coprime}): if $m(\sigma)$ denotes the number of cycles in the disjoint cycle decomposition of a permutation $\sigma \in S_n$, and $\ell$ is a positive integer, then the class function $\sigma \mapsto \ell^{m(\sigma) - 1}$ is a character of $S_n$ if and only if $(n, \ell) = 1$.

For simplicity of exposition, all manifolds herein will be considered complex. However, most of the results are readily extended to the case of smooth algebraic varieties over other algebraically closed fields, provided Brauer traces are used in positive characteristic (see \cite{Donovan}). The results presented here greatly simplify, as well as expand on, part of the Ph.D. thesis \cite{thesis} of the author, who wishes to express his gratitude to E. Goren, F. Bergeron, S. Wang, B. Edixhoven, and the reviewer for their valuable suggestions and comments.

\section{Holomorphic Lefschetz formulae for complete intersections}

In this section we first recall the formal framework in which the Lesfchetz-Riemann-Roch formula (Theorem \ref{th_LRR}) may be stated, and from it derive a generating series expression for the traces of the morphisms induced by an automorphism of a smooth complete intersection satisfying certain properties on the pieces of the Hodge decomposition (Corollary \ref{cor_LRR}).

\subsection*{Vector bundles on pairs}

Consider an automorphism $\sigma$ of finite order of a smooth projective variety $X$; the couple $\XX = (X, \sigma)$ will be called a \emph{smooth projective pair}. By definition, a \emph{vector bundle} on $\XX$ is a pair $\xi = (\VV, \phi)$, where $\VV$ is a vector bundle on $X$ and $\phi: \sigma^* \VV \to \VV$ is a vector bundle morphism above $\sigma$. We call \emph{Lefschetz characteristic of $(\XX, \xi)$} the trace (in the graded sense) of the endomorphism induced by $\sigma$ and $\phi$ on the cohomology of $X$ with coefficients in $\VV$. In symbols,
$$ \chi(\XX, \xi) \df \sum_{i=0}^n (-1)^i \tr \big( (\sigma, \phi)^*, \h^i(X, \VV) \big), $$
where $n$ is the (complex) dimension of $X$ and $(\sigma, \phi)^*$ denotes the composite map
$$ \h^\gr(X, \VV) \stackrel{\sigma^*}{\To} \h^\gr(X, \sigma^* \VV) \stackrel{\phi}{\To} \h^\gr(X,\VV). $$

Let $\M(\XX)$ be the ring of \emph{virtual vector bundles} on $\XX$, i.e., the Grothendieck group of the category of vector bundles on $\XX$ endowed with the ring structure coming from the tensor product. Since short exact sequences of bundles induce long exact sequences in cohomology, the Lefschetz characteristic may be viewed as an additive map
$$ \chi(\XX, \, \cdot \,): \M(\XX) \To \CC. $$

If $\Omega^p_X$ denotes the vector bundle whose sections are the regular $p$-forms on $X$, the automorphism $\sigma \! : X \to X$ induces by pullback a map $\der^p\sigma \! : \sigma^* \Omega^p_X \to \Omega^p_X$. By abuse of notation, we will also use $\Omega^p_X$ to denote the pair $(\Omega^p_X, \der^p \sigma)$ considered as an element of $\M(\XX)$. Following conventional usage, the $p$-characteristic and $y$-polynomial of an element $\xi \in \M(\XX)$ are defined as
\begin{equation} \label{y-char}
\chi^p(\XX, \xi) \df \chi(\XX, \Omega^p_X \cdot \xi), \qquad \chi_y(\XX, \xi) \df \sum_{p = 1}^n \chi^p(\XX, \xi) \, y^p \in \CC[y].
\end{equation}

If $Z$ is a subvariety of $X$ that is stable under the action of $\sigma$, then its normal bundle $\NN_{X/Z}$, defined by the short exact sequence
\begin{equation} \label{normal}
0 \To T_* Z \To (T_* X) |_Z \To \NN_{X/Z} \To 0,
\end{equation}
also inherits a natural action coming from the differential of $\sigma$. We will again abuse notation and denote $\NN_{X/Z}$ the corresponding element of $\M(\mathcal{Z})$, where $\mathcal{Z}$ stands for the projective pair $(Z, \sigma|_Z)$.

\subsection*{The Lefschetz-Riemann-Roch formula}

For a smooth projective pair with trivial automorphism $\XX = (X, 1_X)$, we convene to write $\M(X)$ instead of $\M(\XX)$. The \emph{Chern trace} is a ring homomorphism
$$ \ct: \M(X) \To \CC \otimes \A(X) $$
with values in the complex-valued Chow ring of $X$. Its value on a pair $(\VV, \phi)$ can be expressed in terms of the Chern character as
$$ \ct(\VV, \phi) = \sum_{\alpha \in \CC} \alpha \ch(\VV_\alpha) $$
if $\VV = \oplus_\alpha \VV_\alpha$ is the decomposition of $\VV$ in generalized eigenspaces for $\phi$ (see \cite[\S 3]{Donovan}).

Also, the \emph{Todd characteristic} of a virtual bundle $\xi \in \M(X)$ is defined as
$$ \Tt(X, \xi) \df \int_X \td(X) \ct(\xi), $$
where the integration sign means taking the intersection product with the characteristic class of $X$.

\begin{theorem}[Lefschetz-Riemann-Roch formula] \label{th_LRR}

For every smooth projective pair $\XX = (X, \sigma)$, the fixed locus $X_\sigma$ of $\sigma$ is a smooth algebraic subset of $X$, its virtual conormal bundle
\begin{equation} \label{lambda_N}
\lambda(\NN_{X/X_\sigma}) \df \sum_i (-1)^i \wedge^i \! \NN_{X/X_\sigma}^{\ \vee}
\end{equation}
is invertible in $\M(X_\sigma)$, and we have
\begin{equation} \label{LRR}
\chi(\XX,\xi) = \Tt \big( X_\sigma, \lambda(\NN_{X/X_\sigma})^{-1} \cdot \xi|_{X_\sigma} \big).
\end{equation}

\end{theorem}

Expanding the right-hand side of (\ref{LRR}) as a sum over the connected components of the fixed locus, one recovers the formulation of \cite[Cor. 5.5]{Donovan}, which reduces in the case of curves to the celebrated Chevalley-Weil formula. On the other hand, equation (\ref{LRR}) in the particular case where $\XX = (X, 1_X)$ and $\xi = (\VV, 1_\VV)$ is the classical Grothendieck-Riemann-Roch formula (see \cite{Borel-Serre}):
\begin{equation} \label{GRR}
\chi(X, \VV) = \Tt(X, \VV).
\end{equation}

\subsection*{Formal vector bundles}

Given a smooth projective pair $\XX$, it will be convenient to consider elements of the polynomial ring $\M(\XX)[y]$, which we call \emph{formal (virtual) vector bundles}. We extend the definitions of the Lefschetz characteristic and Todd trace to formal vector bundles coefficient-wise and consider the result as an element of the polynomial ring $\CC[y]$. For example, if we define the formal vector bundle
\begin{equation} \label{omega_gr} \Omega_X^\bullet \df \sum_{p=0}^n \Omega_X^p \, y^p, \end{equation}
then the $y$-polynomial (\ref{y-char}) of a virtual vector bundle $\xi \in \M(\XX)$ may be expressed as the formal Lefschetz characteristic
\begin{equation} \label{y-char_again} \chi_y(\XX, \xi) = \chi(\XX, \Omega_X^\bullet \cdot \xi). \end{equation}

Also, for $\xi \in \M(\XX)$, consider the associated formal vector bundle defined by
$$ \lambda_y(\xi) \df \sum_p (\wedge^p \xi^\vee) \, y^p, $$
where $\xi^\vee$ denotes the dual of $\xi$. It is readily checked that for $\xi, \nu \in \M(\XX)$ we have
\begin{equation} \label{lambda_y}
\lambda_y(\xi + \nu) = \lambda_y(\xi) \cdot \lambda_y(\nu).
\end{equation}
Consistently with (\ref{lambda_N}), we put $\lambda(\xi) \df \lambda_{-1}(\xi)$; provided
$\lambda(\xi)$ is invertible in $\M(\XX)$, we also define $ \widetilde{\lambda}_y(\xi) \df \lambda(\xi)^{-1} \cdot \lambda_y(\xi)$.

\subsection*{Polynomial vector bundles on complete intersections}

Let $Z$ be a smooth projective complete intersection with a fixed embedding $Z \incl \PP^m$ of multidegree $\dd = (d_1, \ldots, d_r)$ in a projective space (we allow $r = 0$, in which case $Z = \PP^m$). Let $\HH$ be the restriction to $Z$ of the line bundle on $\PP^m$ associated to its hyperplane class. We will say that a virtual vector bundle $\xi \in \M(Z)$ is \emph{weakly polynomial} if there exists a Laurent polynomial $f \in \CC[h, h^{-1}]$ for which
$$ \ct \xi = f( \ch \HH ). $$
In that case, the Laurent polynomial $f$ will be said to be \emph{associated} to $\xi$ (note that associated polynomials are not unique). If $\xi$ lies in the subring of $\M(Z)$ additively generated by the pairs of the form $(\HH^d, \alpha)$, with $\alpha \in \CC$, $d \in \ZZ$,
then $\xi$ will be said to be \emph{strongly polynomial}. Strongly polynomial clearly implies weakly polynomial since $(\HH^d, \alpha)$ admits the monomial $\alpha h^d$ as associated polynomial.

To the multidegree $\dd = (d_1, \ldots, d_r)$ we associate the power series in two variables
\begin{equation} \label{Phi_d} \Phi_\dd(x,y) \df \frac{1}{(1 + x y) (1 - x)} \prod_{i=1}^r \frac{(1 + x y)^{d_i} - (1-x)^{d_i}}{(1 + x y)^{d_i} + y(1-x)^{d_i}}. \end{equation}
Also, if $\Phi$ is a power series in $x$, the notation $[x^m]\{ \Phi \}$ will be used to denote its $m^\ieme$ coefficient.

\begin{lemma} \label{lemma_LRR}

Let $Z$ be a smooth complete intersection of multidegree $\dd$ and $\HH$ the restriction to $Z$ of the hyperplane line bundle on the ambient projective space $\PP^m$. For any Laurent polynomial $f \in \CC[h,h^{-1}]$, we have
$$ \chi_y(Z, f(\HH)) = [x^m] \left\{ \Phi_\dd(x,y) \, f \! \! \left( \frac{1 + x y}{1 - x} \right) \right\}. $$

\end{lemma}

\begin{proof}

That this statement holds for powers of $h$ is the content of \cite[Th. 22.1.1]{Hirzebruch}, and the general case follows by linearity. \end{proof}

\subsection*{Main formula}

We are now ready to prove the generating series expression for the trace of the induced morphism in cohomology that will be used in the sequel.

\begin{theorem} \label{main_formula}

Let $\XX = (X, \sigma)$ be a smooth projective pair and $\xi \in \M(\XX)$. Suppose that every connected component $Z$ of the fixed locus $X_\sigma$ admits a projective embedding $Z \incl \PP^{m_Z}$ for which:
\begin{itemize}
\item $Z$ is a smooth complete intersection of multidegree $\dd_Z$,
\item the virtual bundle $\widetilde{\lambda}_y(\NN_{X/Z})$ is weakly polynomial, and
\item the restriction $\xi|_{Z}$ is weakly polynomial.
\end{itemize}
Then
$$ \chi_y(\XX, \xi) = \sum_Z [x^{m_Z}] \left\{ \Phi_{\dd_Z}(x,y) \, f_Z \bigg( \frac{1 + x y}{1 - x}, y \bigg) \, g_Z \bigg( \frac{1 + x y}{1 - x} \bigg) \right\}, $$
where $g_Z(h)$ and $f_Z(h,y)$ are the Laurent polynomials associated as above to $\xi|_Z$ and $\widetilde{\lambda}_y(\NN_{X/Z})$, respectively.

\end{theorem}

\begin{proof}

Applying formally the Lefschetz-Riemann-Roch formula (\ref{LRR}) to the right-hand side of (\ref{y-char_again}), we find
\begin{equation} \label{first_step} \chi_y(\XX, \xi) = \sum_Z \Tt \big(Z, \lambda(\NN_{X/Z})^{-1} \cdot \Omega_X^\gr|_Z \cdot \xi|_Z \big), \end{equation}
the sum being taken over the set of connected components $Z$ of $X_\sigma$. Note that the formal vector bundle $\Omega_X^\gr$, as defined by (\ref{omega_gr}), may be interpreted as $\lambda_y(T_* X)$. But according to the short exact sequence (\ref{normal}), for every connected component of $X_\sigma$ the identity
$$ T_* X|_Z  = T_* Z + \NN_{X/Z} $$
holds in $\M(Z)$; hence by (\ref{lambda_y}) we have
$$ \Omega_X^\gr|_Z = \lambda_y(T_* X)|_Z = \lambda_y(T_* Z) \cdot \lambda_y(\NN_{X/Z}) = \Omega_Z^\gr \cdot \lambda_y(\NN_{X/Z}). $$
Using this and applying formally the Grothendieck-Riemann-Roch formula (\ref{GRR}), the contribution of $Z$ to the right-hand side of (\ref{first_step}) can we written as
$$ \Tt \big(Z, \widetilde{\lambda}_y(\NN_{X/Z}) \cdot \Omega_Z^\gr \cdot \xi|_Z \big) = \chi_y(Z, \widetilde{\lambda}_y(\NN_{X/Z}) \cdot \xi|_Z). $$
But we are given that $\widetilde{\lambda}_y(\NN_{X/Z}) \cdot \xi|_Z$ is a weakly polynomial formal vector bundle with associated Laurent polynomial $f_Z(h,y) \, g_Z(h)$, hence the conclusion follows by Lemma \ref{lemma_LRR}. \end{proof}

The hypothesis in Theorem \ref{main_formula} about the normal bundles is automatically satisfied whenever they are strongly polynomial. Indeed, the following lemma follows directly from the multiplicativity of $\widetilde{\lambda}_y$.

\begin{lemma} \label{strongly_pol_lemma}

Let $Z$ be a complete intersection and $\xi \in \M(Z)$ a strongly polynomial bundle such that $\lambda(\xi)$ is invertible. Then the virtual bundle $\widetilde{\lambda}_y(\xi)$ is also strongly polynomial; explicitly, if $\xi = \sum_i m_i (\HH^{e_i}, \alpha_i) $ with $m_i, e_i \in \ZZ$ and $\alpha_i \in \CC^\times$, then
$$ \widetilde{\lambda}_y(\xi) = \prod_i \left( \frac{1 + (\HH^{-e_i}, \alpha_i) \, y}{1 - (\HH^{-e_i}, \alpha_i)} \right)^{m_i}, $$
hence admits as associated Laurent polynomial
$$ f_\xi(h,y) = \prod_i \left( \frac{h^{e_i} + \alpha_i \, y}{h^{e_i} - \alpha_i} \right)^{m_i}. $$
\end{lemma}

Let us introduce new notation in order to restate the main formula for the special case of interest in next section. Define, for $\gamma \in \CC^\times$ and $e \geq 0$,
$$ \Psi_{\gamma, e}(x,y) \df \frac{ \gamma (1 + x y)^e + y (1 - x)^e } { \gamma (1 + x y)^e - (1 - x)^e }, $$
so that the formal series $\Phi_\dd$ of (\ref{Phi_d}) may be expressed as
\begin{equation} \label{Phi_Psi} \Phi_\dd(x,y) = \frac{1}{(1+x y) (1 - x)} \prod_{i=1}^r \frac{1}{\Psi_{1,d_i}(x,y)}. \end{equation}
Then, if $f_\xi(h,y)$ denotes the polynomial given in Lemma \ref{strongly_pol_lemma} above, we have
\begin{equation} \label{F_Z} F_\xi \df f_\xi \left( \frac{1 + xy}{1 - x}, y \right) = \prod_i \Psi_{1/\alpha_i, e_i}(x,y)^{m_i}. \end{equation}
Finally, by a slight abuse of notation we now write $\chi_y(X, \sigma)$ instead of $\chi_y(\XX,\CC)$ when the bundle above the projective pair $\XX = (X, \sigma)$ is trivial.

\begin{corollary} \label{cor_LRR}

Let $(X, \sigma)$ be a smooth projective pair such that every connected component $Z$ of the fixed locus $X_\sigma$ admits a projective embedding $Z \incl \PP^{m_Z}$ for which $Z$ is a complete intersection of multidegree $\dd_Z$ and $\NN_{X/Z}$ is strongly polynomial. Then
$$ \chi_y(X, \sigma) = \sum_Z [x^{m_Z}] \big\{ \Phi_{\dd_Z}(x,y) \, F_Z (x,y) \big\}, $$
where $F_Z(x,y)$ is the power series associated to $\xi = \NN_{X/Z}$ as in (\ref{F_Z}).

\end{corollary}

\section{Smooth projective hypersurfaces}

We now restrict our attention to the case where the automorphism $\sigma$ extends to a projective transformation of the ambient projective space.

\subsection*{Fixed points of projective transformations}

Let $V$ be a finite-dimensional vector space and consider a projective transformation $\sigma \in \PGL(V)$ of finite order acting on $\PP = \PP(V)$. Choosing a linear representative $\widetilde{\sigma} \in \GL(V)$ for $\sigma$, the fixed locus of $\sigma$ in $\PP$ decomposes as
$$ \PP_\sigma = \bigsqcup_{\alpha \in \CC^\times} \PP_\alpha, \qquad \PP_\alpha = \PP(V_\alpha), $$
where $V_\alpha$ is the $\alpha$-eigenspace of $\widetilde{\sigma}$ in $V$. Note that the hypothesis that $\sigma$ has finite order implies that $\widetilde{\sigma}$ is diagonalizable, hence $V = \bigoplus_\alpha V_\alpha$.

Recall that the tangent bundle of $\PP$ may be written as
$$ T_* \PP = \HH \otimes \underline{V}, $$
where $\HH$ denotes the hyperplane line bundle on $\PP$ and $\underline{V}$ the trivial vector bundle with fiber $V$. Using the corresponding fact for a fixed component $\PP_\alpha$ and the short exact sequence (\ref{normal}), we find that the normal bundle of $\PP_\alpha$ may be written as
$$\NN_{\PP/\PP_\alpha} = \HH_\alpha \otimes \underline{V/V_\alpha} = \bigoplus_{\beta \neq \alpha} \HH_\alpha \otimes \underline{V_\beta}, $$
where $\HH_\alpha$ denotes the restriction of $H$ to $\PP_\alpha$. Taking the endomorphism induced by $\sigma$ into account, we obtain in $\M(\PP_\alpha)$ the decomposition
\begin{equation} \label{normal_PP_alpha} \NN_{\PP/\PP_\alpha} = \sum_{\beta \neq \alpha} m_\beta \bigg( \HH_\alpha, \frac{\beta}{\alpha} \bigg), \end{equation}
where $m_\beta = \dim V_\beta$ is the multiplicity of the eigenvalue $\beta$ of $\widetilde{\sigma}$. Note that the description (\ref{normal_PP_alpha}) does not depend on the choice of the representative $\widetilde{\sigma}$ of $\sigma$.

\begin{lemma} \label{lemma_P}

If $\dim V = n + 2$ and $\sigma \in \PGL(V)$ has finite order, we have
$$ \chi_y(\PP, \sigma) = \sum_{\alpha \in \CC^\times} [x^{m_\alpha - 1}] \bigg\{ \frac{1}{(1 + x y)(1-x)} \prod_{\beta \neq \alpha} \Psi_{\alpha/\beta,1}(x,y) \bigg\} = \frac{1 - (-y)^{n+2}}{1 + y}, $$
where $\widetilde{\sigma} \in \GL(V)$ is any lift of $\sigma$.

\end{lemma}

\begin{proof}

The left-hand side is the expression for $\chi_y(\PP,\sigma)$ given by Corollary \ref{cor_LRR} in which (\ref{normal_PP_alpha}) was used. On the other hand, $\sigma$ preserves the hyperplane class of $\PP$, hence acts trivially on its cohomology. We may thus evaluate directly
$$ \chi_y(\PP,\sigma) = \chi_y(\PP) = \sum_{p=0}^{n+1} \chi^p(\PP) \, y^p = \sum_{p=0}^{n+1} (-y)^p = \frac{1 - (-y)^{n+2}}{1 + y}, $$
from which the conclusion follows. \end{proof}

\subsection*{Smooth hypersurfaces}

With the notations of the previous section, suppose that $X$ is a smooth hypersurface in $\PP$ that is stable under the action of a projective transformation $\sigma$. If $X$ is defined by a homogeneous polynomial $f$ of degree $d$, we have $f \circ \widetilde{\sigma} = \lambda f$ for a certain constant $\lambda \neq 0$. Choosing the representative $\widetilde{\sigma}$ of $\sigma$ appropriately, we may always assume that $\lambda = 1$, i.e., that the defining polynomial $f$ is invariant under $\widetilde{\sigma}$.

\begin{theorem} \label{th_main}

Let $X$ be a smooth hypersurface of degree $d$ in $\PP(V)$ that is stable under the action of a projective transformation $\sigma \in \PGL(V)$ of finite order. The formal Lefschetz characteristic of $\sigma$ acting on $X$ is given by
$$ \chi_y(X,\sigma) = \sum_{\alpha \in \CC^\times} [x^{m_\alpha-1}] \bigg\{ \frac{1}{(1 + x y) (1 - x)} \cdot \frac{1}{\Psi_{\alpha^d, d}(x,y)} \cdot \prod_{\beta \neq \alpha} \Psi_{\alpha/\beta,1}(x,y)^{m_\beta} \bigg\}, $$
where $\widetilde{\sigma} \in \GL(V)$ is the lift of $\sigma$ for which a defining polynomial for $X$ is invariant, and $\{m_\alpha\}_{\alpha \in \CC}$ are its eigenvalue multiplicities.

\end{theorem}

The fixed locus of the smooth projective pair $\XX = (X, \sigma)$ decomposes as
$$ X_\sigma = \bigsqcup_{\alpha \in \CC^\times} X_\alpha, \qquad X_\alpha = X \cap \PP_\alpha. $$
To prove the theorem, we need to consider the normal bundle of $X_\alpha$ in $X$, which may be described indirectly by considering the following diagram of inclusions.
$$ \xymatrix{ X \ar[r] & \PP \\ X_\alpha \ar[u] \ar[r] & \PP_\alpha \ar[u] } $$
Using the appropriate short exact sequences corresponding to (\ref{normal}), one obtains in $\M(X_\alpha)$ the equality
\begin{equation} \label{sum_normal} \NN_{X/X_\alpha} = \NN_{\PP/\PP_\alpha} |_{X_\alpha} + \NN_{\PP_\alpha/X_\alpha} - \NN_{\PP/X}|_{X_\alpha}. \end{equation}

Concerning $X_\alpha \incl \PP_\alpha$, we have the following dichotomy.

\begin{lemma} \label{new_lemma}

If $m_\alpha > 0$ and $\alpha^d = 1$, then $X_\alpha$ is a smooth hypersurface of degree $d$ in $\PP_\alpha$; in all other cases, $X_\alpha = \PP_\alpha$.

\end{lemma}

\begin{proof}

We may think of $f$ as being expressed in the coordinates of an eigenbasis of $V$ for $\widetilde{\sigma}$, hence write it as
$$ f = f_\alpha + g, $$
where $f_\alpha$ is the sum of monomials of $f$ involving only variables corresponding to the eigenspace $V_\alpha$, and $g$ comprises the mononomials which involve at least one variable corresponding to some other eigenvalue $\beta \neq \alpha$.

The closed subset $X_\alpha$ of $\PP_\alpha$ is defined by the homogeneous polynomial $f_\alpha = f|_{V_\alpha}$. Note that since we assume that $f$ in $\widetilde{\sigma}$-invariant, we have
$$ f_\alpha = f_\alpha \circ \widetilde{\sigma}|_{V_\alpha} = \alpha^d f_\alpha, $$
hence either $\alpha^d = 1$ or $f_\alpha = 0$. We need to show that both conditions cannot be simultaneously satisfied when $m_\alpha > 0$.

Suppose that $\alpha^d = 1$. Then note that the invariance of $f$ under $\widetilde{\sigma}$ implies that all monomials in $g$ actually involve \emph{at least two} variables corresponding to eigenvalues other that $\alpha$. Using the Jacobian criterion, we then see that if $f_\alpha = 0$, all $\PP_\alpha \neq \varnothing$ occurs as singularities of $X$, which is impossible since $X$ was supposed smooth.

It remains to show that $X_\alpha$ is actually smooth when $f_\alpha \neq 0$. But again, it is easily seen from the Jacobian criterion that all singularities of $X_\alpha$ are also singular points for $X$, hence there can be none (a different proof is given in \cite[4.1]{Donovan}). \end{proof}

We may now proceed with the proof of Theorem \ref{th_main}.

\begin{proof}[Proof of the theorem]

The normal bundle of the linear space $\PP_\alpha$ was described by formula (\ref{normal_PP_alpha}) above; let us now consider the two other terms in (\ref{sum_normal}).

Let $f$ be a defining polynomial for $X$ that is invariant under $\widetilde{\sigma}$. If $\der f: \underline{V} \to \underline{V}$ denotes the vector bundle morphism over $\Aff^{n+1}$ given on the fiber above $x \in V$ by the differential $\der f_x$, then the normal bundle of $X$ in $\PP$ may be identified with
$$ \HH|_X \otimes \big( \underline{V}/\Ker \der f \big) = \HH|_X \otimes \HH^{d-1}|_X = \HH^d|_X. $$
Restricting to $X_\alpha$ and taking the action of $\sigma$ into account, we find
\begin{equation} \label{norm_hyper} \NN_{\PP/X}|_{X_\alpha} = \bigg( \HH^d, \frac{1}{\alpha^d} \bigg). \end{equation}

Now consider the two possibilities for $X_\alpha \incl \PP_\alpha$ described in Lemma \ref{new_lemma}.

\begin{itemize}

\item When $m_\alpha = 0$ or $\alpha^d \neq 1$, then $X_\alpha = \PP_\alpha$, so that (\ref{sum_normal}) reduces to
$$ \NN_{X/X_\alpha} = \NN_{\PP/\PP_\alpha} - \NN_{\PP/X}|_{X_\alpha} = \sum_{\beta \neq \alpha} m_\beta \bigg( \HH_\alpha, \frac{\beta}{\alpha} \bigg) - \bigg( \HH^d, \frac{1}{\alpha^d} \bigg). $$
The contribution of $X_\alpha$ to the right-hand side of the formula for the $y$-characteristic (Corollary \ref{cor_LRR}) is thus
$$ [x^{m_\alpha - 1}] \bigg\{ \Phi_\emptyset(x,y) \cdot \frac{1}{\Psi_{\alpha^d, d}(x,y)} \cdot \prod_{\beta \neq \alpha} \Psi_{\alpha/\beta,1}(x,y)^{m_\beta} \bigg\}, $$
which is precisely what appears in the statement of the theorem.

\item When $m_\alpha > 0$ and $\alpha^d = 1$, then $X_\alpha$ is a smooth hypersurface of degree $d$ inside $\PP_\alpha$. As for (\ref{norm_hyper}), we have
$$ \NN_{\PP_\alpha/X_\alpha} = \bigg( \HH^d, \frac{1}{\alpha^d} \bigg) = \big( \HH^d, 1 \big). $$

The contribution of $X_\alpha$ it thus this time
$$ [x^{m_\alpha - 1}] \bigg\{ \Phi_d(x,y) \cdot \frac{1}{\Psi_{1, d}(x,y)} \cdot \prod_{\beta \neq \alpha} \Psi_{\alpha/\beta,1}(x,y)^{m_\beta} \cdot \Psi_{1,d}(x,y) \bigg\}. $$

\end{itemize}

Using (\ref{Phi_d}), we see that the contributions of $X_\alpha$ in both cases agree with the announced formula. \end{proof}

It should be noted that the formula of Theorem \ref{th_main} expresses the trace of the action of $\sigma$ on the cohomology of $X$ in terms only of the degree of $X$ and the spectrum of $\widetilde{\sigma}$.

\subsection*{Traces on primitive cohomology}

Let $\chi_y^\prim(X,\sigma)$ denote the formal alternating trace of the morphism induced by $\sigma$ on the primitive cohomology of a smooth hypersurface $X$ of dimension $n$, i.e.,
$$ \chi_y^\prim(X,\sigma) = \sum_{p + q = n} (-1)^p \, \tr \big(\sigma^*, \h^{p,q}_\prim(X,\CC) \big) \,  y^p. $$

\begin{corollary} \label{cor_prim}

With the notations of Theorem \ref{th_main}, the formal trace $\chi_y^\prim(X,\sigma)$ of $\sigma$ on the primitive cohomology of $X$ is given by
$$  y^{n+1} + \\ (-1)^n \sum_{\alpha \in \CC^\times} [x^{m_\alpha-1}] \bigg\{ \frac{1}{(1 + x y) (1 - x)} \left( \frac{1}{\Psi_{\alpha^d, d}(x,y)} - 1 \right) \prod_{\beta \neq \alpha} \Psi_{\alpha/\beta,1}(x,y)^{m_\beta} \bigg\} . $$

\end{corollary}

\begin{proof}

The Lefschetz hyperplane theorem implies that
$$ \chi_y(X,\sigma) = \chi_y(\PP^n) + \chi_y^\prim(X,\sigma). $$
But since $ \chi_y(\PP^n) = \chi_y(\PP) - y^{n+1}, $
this may be rewritten as
$$ (-1)^n \chi_y^\prim(X,\sigma) = \chi_y(X,\sigma) - \chi_y(\PP,\sigma) + y^{n+1}, $$
and the conclusion follows by using the expressions for $\chi_y(\PP,\sigma)$ and $\chi_y(X,\sigma)$ given by Lemma \ref{lemma_P} and Theorem \ref{th_main}, respectively. \end{proof}

Specializing this formula at $y = -1$ yields a formula for the trace of $\sigma$ on the whole primitive cohomology of $X$. Note that this formula could also have been obtained using the ``classical'' Lefschetz formula in terms of the Euler characteristic of the fixed locus,
$$ \chi(X,\sigma) = \chi(X_\sigma) = \sum_\alpha \chi(X_\alpha), $$
and the description of $X_\alpha$ given by Lemma \ref{new_lemma}.

\begin{corollary} \label{cor_nice}

If $X$ is a smooth degree $d$ projective hypersurface of dimension $n$ stable under the action of a projective transformation $\sigma$ of finite order, we have
$$ \tr(\sigma^*, \h^n_\prim(X,\CC)) = \frac{(-1)^n}{d} \sum_{\alpha^d = 1} (1 - d)^{m_\alpha}, $$
where $m_\alpha$ is the multiplicity of $\alpha$ as an eigenvalue of the linear representative of $\sigma$ which leaves invariant a defining polynomial for $X$.

\end{corollary}

The proof is an easy exercise in generating series using the following fact: if $\gamma \neq 1$, we have $\Psi_{\gamma,e}(x,-1) = 1$, while using the formal L'Hospital rule we find
$$ \Psi_{1,e}(x,-1) =  \frac{1 + (e - 1) x}{e x}. $$

Note that for $\sigma = 1$, we recover from Corollary \ref{cor_nice} the well-known formula for the dimension of the primitive cohomology
$$ \dim \h^n_\prim(X,\CC) = \frac{(d-1)^{n+2} + (-1)^n (d-1)}{d}. $$

\subsection*{Smooth symmetric hypersurfaces}

For $n$ a fixed integer, consider the standard permutation representation of the symmetric group $S_{n+2}$ on $V = \CC^{n+2}$. For every $\sigma \in S_{n+2}$, the characteristic polynomial of $\sigma$ acting on $V$ is $ \prod_i (x^{\lambda_i} - 1) $ if the $\lambda_i$ are the lengths of the cycles in the disjoint cycle decomposition of $\sigma$. If $\alpha \in \CC^\times$ is a primitive $e^\ieme$ root of unity, it follows that $\alpha$ occurs as an eigenvalue of $\sigma$ exactly once for every cycle whose length $\lambda_i$ is divisible by $e$. Accordingly, let
$ m_e(\sigma)$ denote the number of cycles of $\sigma$ whose length is divisible by $e$.

Let $X$ be a smooth hypersurface of degree $d$ in $H = \PP(V)$ that is stable under the projective action of $S_{n+2}$ (e.g., the Fermat hypersurface of degree $d$). It is easily checked that for $n \geq 1$, any such hypersurface is actually defined by a \emph{symmetric} polynomial; hence no confusion should arise if we use the same notation $\sigma$ for both the projective and the linear transformation corresponding to a permutation $\sigma \in S_n$. According to Corollary \ref{cor_nice}, the character of $S_{n+2}$ acting on the primitive cohomology of $X$ is
\begin{equation} \label{type_I} \chi(\sigma) = \frac{(-1)^n}{d}  \sum_{e \mid d} \varphi(e)(1 - d)^{m_e(\sigma)}, \end{equation}
where $\varphi(e)$ is the Euler totient function.

Now considering $V$ as the affine hyperplane $x_1 + \cdots + x_{n+3} = 0$ in $\CC^{n+3}$, we see that it admits an action by $S_{n+3}$ compatible with the action of $S_{n+2}$ described above; this is the standard \emph{irreducible} permutation representation of the symmetric group. As above, it is still the case that any smooth $S_{n+3}$-invariant hypersurface in $H$ is defined by an invariant polynomial. In fact, every such hypersurface occurs as the intersection of $H$ with a (possibly singular) symmetric hypersurface in $\PP^{n+2}$ (see \cite[3.2.8]{thesis}). Consequently, if a smooth $n$-dimensional degree $d$ hypersurface $X$ is stable under this action, it follows from Corollary \ref{cor_nice} applied to $X \subset H$ that the character of $S_{n+3}$ on its primitive cohomology is given by
\begin{equation} \label{type_II} \psi(\sigma) = \frac{(-1)^{n+1}}{d}  \sum_{e \mid d} \varphi(e)(1 - d)^{m'_e(\sigma)}, \end{equation}
where $m'_e(\sigma)$ is a slight variation on $m_e(\sigma)$ defined as
$$ m'_e(\sigma) \df \begin{cases} \ \ m_e(\sigma) & \text{for } e > 1, \\ m_1(\sigma) - 1 & \text{for } e = 1. \end{cases} $$

Comparing (\ref{type_I}) and (\ref{type_II}), we obtain the following result.

\begin{theorem} \label{th_existence}

There exists a smooth projective hypersurface of dimension $n \geq 1$ and degree $d \geq 2$ stable under the standard irreducible permutation action of $S_{n+3}$ if and only if $n+3$ and $d - 1$ are coprime.

\end{theorem}

\begin{proof}

If $(n+3, d-1) = 1$, it is readily checked using the Jacobian criterion that the intersection of the Fermat hypersurface of degree $d$ in $\PP^{n+2}$ with $\PP^{n+1}$ is smooth; hence one part of the statement is clear. Now suppose that $X$ is a smooth $S_{n+3}$-stable hypersurface in $\PP^{n+1}$; the character of the representation of $S_{n+3}$ on its primitive cohomology is given by \ref{type_II}. But since (\ref{type_I}) also describes a character of $S_{n+3}$, it follows that
$$ \chi(\sigma) + \psi(\sigma) = (-1)^{n+1} (1 - d)^{m_1(\sigma) - 1} = \sg(\sigma) \, (d - 1)^{m_1(\sigma) - 1} $$
also is a character of $S_{n+3}$. But this is possible only if $(n+3, d-1) = 1$, as will be proven in the next section. \end{proof}

\section{A representation-theoretic fact}

This section is devoted to a proof of the following fact.

\begin{theorem} \label{coprime} For $n$ and $\ell$ positive integers, the class function $\sigma \mapsto \ell^{m_1(\sigma) - 1}$ on $S_n$ is a character if and only if $n$ and $\ell$ are coprime. \end{theorem}

\subsection*{An abelian group construction}

Let $A$ be a finite abelian group of order $\ell$, and consider the $\ZZ$-linear representation of $S_n$ on $A^n$ by permutation of coordinates. Let $M$ be the subgroup of $A^n$ where the sum of the coordinates is $0$; we have a short exact sequence of $\ZZ[S_n]$-modules \begin{equation} \label{sex} 0 \To M \To A^n \stackrel{\Sigma}{\To} A \To 0, \end{equation}
where $A$ has the trivial action. For $\sigma \in S_n$, let $d(\sigma)$ denote the greatest common divisor of the lengths of its disjoint cycles and $d_A(\sigma)$ the cardinality of the kernel of the multiplication map
$ d(\sigma): A \to A. $

\begin{proposition} \label{prop_M} The character of the linear representation $\CC[M]$ is given for $\sigma \in S_n$ by
$$ \chi_M(\sigma) = d_A(\sigma) \cdot \ell^{m_1(\sigma) - 1}. $$

\end{proposition}

\begin{proof}

The value of $\chi_M(\sigma)$ is the number of fixed points of $\sigma$ in $M$. Let $\{ \lambda_1, \ldots, \lambda_m \}$ be the partition associated to $\sigma$. Without loss of generality, we may assume that $\sigma$ lies in the image of $S_{\lambda_1} \times \cdots \times S_{\lambda_m}$ in $S_n$ under the standard embedding; then the set of fixed points of $\sigma$ in $A^n$ is
$$ \{ ( \underbrace{a_1, \ldots, a_1}_{\lambda_1}, \ \ldots, \ \underbrace{a_m, \ldots, a_m}_{\lambda_m} ) \mid a_1, \ldots, a_m \in A \}, $$
which is a subgroup isomorphic to $A^m$. The set of fixed points of $\sigma$ in $M$ corresponds to the kernel of the linear map
$$ \lambda: A^m \To A, \quad (a_1, \ldots, a_m) \mapsto \sum_{i=1}^m \lambda_i a_i. $$
By standard linear algebra over $\ZZ$, this map is equivalent to the one defined by the matrix $ (d(\sigma), 0, \ldots, 0), $
hence the conclusion follows. \end{proof}

Note that when $n$ and $\ell$ are coprime, multiplication by $d(\sigma)$ is invertible on $A$ for any $\sigma \in S_n$, since $d(\sigma) \mid n$. Hence in that situation, we have
$$ \chi_M(\sigma) = \ell^{m_1(\sigma) - 1}, $$
establishing half of Theorem \ref{coprime}. In that case, the short exact sequence (\ref{sex}) is also easily seen to be split, so that
$$ \CC[A^n] \iso \CC[M] \otimes \CC[A] \iso \CC[M]^{\ell}. $$

\subsection*{The coprimality condition}

Note that the class function $\theta_{n,\ell}(\sigma) \df \ell^{m_1(\sigma)}$ on $S_n$ always is the character of a representation, namely that of the permutation representation $\CC[A^n]$. Its decomposition into irreducible representations can be described in terms of the Kotska numbers $K_{\lambda \mu}$ as follows. Consider the \emph{shape} of an element $\el \in A^n$ to be the partition of $n$ whose parts count the number of occurrences of the distinct coordinates of $\el$. For $\mu \vdash n$, let $c_\mu(\ell)$ denote the number of orbits for the action of $S_n$ on $A^n$ whose elements have shape $\mu$, i.e.,
$$ c_\mu(\ell) = \frac{\ell (\ell - 1) \cdots (\ell - k + 1)}{\prod_i d_i}, $$
where $k$ is the number of parts of $\mu$ and $d_i$ denotes the number of times that the integer $i$ occurs among them.

\begin{proposition} \label{label_mult}

The multiplicity of the Schur module $V_\lambda$ in $\theta_{n,\ell}$ is given by
$$ \langle V_\lambda, \theta_{n,\ell} \rangle = \sum_{\mu \vdash n} c_\mu(\ell) \, K_{\lambda \mu}, $$
where $K_{\lambda \mu}$ is the number of semistandard Young tableaux of shape $\mu$ and content~$\lambda$.

\end{proposition}

\begin{proof}

The stabilizer in $S_n$ of an element $\el \in A^n$ of shape $\mu$ is conjugated to the Young subgroup
$$ S_\mu \iso S_{\mu_1} \times \cdots \times S_{\mu_k} \incl S_n. $$
Decomposing $A^n$ into orbits for the action of the symmetric group, we thus find
$$ \CC[A^n] \iso \bigoplus_{\mu \vdash n} \CC[S_n/S_\mu]^{c_\mu(\ell)}. $$
The result then follows using Young's rule \cite[Cor. 4.39]{Fulton-Harris}. \end{proof}

For example, the multiplicity of the alternating representation can be found by putting $\lambda = \{1,1,\ldots,1\}$ in Proposition \ref{label_mult}:
$$ \langle \sg, \theta_{n,\ell} \rangle = \langle V_\lambda, \theta_{n,\ell} \rangle = c_\lambda(\ell) = \binom{\ell}{n}. $$
We will also need an explicit formula for the multiplicity of the trivial character.

\begin{lemma}

The multiplicity of the trivial representation in $\theta_{n,\ell}$ is
\begin{equation} \label{mult} \langle \triv, \theta_{n,\ell} \rangle = \sum_{\mu \vdash n} c_\mu(\ell) = \binom{n + \ell - 1}{n}. \end{equation}

\end{lemma}

\begin{proof} Indeed, the left-hand side is the dimension of the space of invariants in $\CC[A^n]$, so this multiplicity is the number of orbits for the action of $S_n$ on $A^n$. The orbits correspond to unordered lists $a_1, \ldots, a_n$ of elements of $A$, with repetition allowed, and the number of these is the number of ways to put $n$ identical balls (the coefficients $a_i$) into $\ell$ boxes (the elements of $A$), which is given by the right-hand side. \end{proof}

We may now prove the second half of Theorem \ref{coprime}.

\begin{proposition}

If $\widetilde{\theta}_{n,\ell} = \ell^{-1} \theta_{n,\ell}$ is a character of $S_n$, then $(n,\ell) = 1$.

\end{proposition}

\begin{proof}

Since characters form a ring, if $\widetilde{\theta}_{n,\ell}$ is a character then $\widetilde{\theta}_{n,\ell^\alpha} = (\widetilde{\theta}_{n,\ell})^\alpha$ also is one for every $\alpha \geq 1$. Equivalently, the multiplicity of every irreducible representation in $\theta_{n,\ell^\alpha}$ is divisible by $\ell^\alpha$ for every $\alpha \geq 1$, in particular that of the trivial representation. The conclusion now follows using the following lemma and the multiplicity formula (\ref{mult}). \end{proof}

\begin{lemma}

If $(n,\ell) \neq 1$, there exists a positive integer $\alpha$ such that
$$ \binom{n + \ell^\alpha - 1}{n} \not\equiv 0 \mod{\ell^\alpha}. $$

\end{lemma}

\begin{proof}

For fixed $n$, consider the polynomial $f(L) = \prod_{i=1}^{n-1} (L + i)$, so that
$$ \binom{n + L - 1}{n} = \frac{f(L) \cdot L}{n!} \equiv 0 \mod{L} \iff f(L) \equiv 0 \mod{n!}. $$
Suppose $p$ is a prime number dividing $(n,\ell)$. The $p$-adic valuation of $n!$ being fixed, for $\alpha$ sufficiently large we have
$ \ell^\alpha \equiv 0 \mod p^{v_p(n!)}, $
so that
$$ f(\ell^\alpha) \equiv f(0) \equiv (n-1)! \not\equiv 0 \mod p^{v_p(n!)}, $$
and the conclusion follows. \end{proof}

\end{document}